\documentclass{ifacconf}

\usepackage{graphicx}      % include this line if your document contains figures
\usepackage{natbib}        % required for bibliography
\usepackage{epsfig} % for postscript graphics files
\usepackage{mathptmx} % assumes new font selection scheme installed
\usepackage{times} % assumes new font selection scheme installed
\usepackage{amsmath} % assumes amsmath package installed
\usepackage{amssymb}  % assumes amsmath package installed
\usepackage{subfigure}
\usepackage{color}
\usepackage{flushend}
\usepackage{eurosym}
\usepackage{theorem}
\usepackage{amsfonts}
\usepackage{psfrag}
\usepackage{array}
\usepackage{shortvrb}
\usepackage{epsf}
\usepackage{rotating}
\usepackage{wasysym}
\usepackage[usenames]{xcolor}

\usepackage{algorithm}
\usepackage{algpseudocode}
\usepackage{pifont}

\makeatletter
\def\BState{\State\hskip-\ALG@thistlm}
\makeatother

\usepackage{mathtools}
\usepackage{amsfonts}
\usepackage{epstopdf}
\usepackage{epsf,subfigure}
\usepackage{psfrag}

\newtheorem{proposition}{Proposition}
\newtheorem{definition}{Definition}

\newtheorem{lemma}{Lemma}

\newenvironment{proof}[1][Proof]{\begin{trivlist}
\item[\hskip \labelsep {\bfseries #1}]}{\end{trivlist}}
\newenvironment{remark}[1][Remark]{\begin{trivlist}
\item[\hskip \labelsep {\bfseries #1}]}{\end{trivlist}}

\begin{document}
\begin{frontmatter}

%\title{Decentralized Control for Exponential Stabilization of Large Nonlinear Networks\thanksref{footnoteinfo}} 
\title{Distributed Coordinated Control of Large-Scale Nonlinear Networks\thanksref{footnoteinfo}} 
% Title, preferably not more than 10 words.

\thanks[footnoteinfo]{{This work was supported by the U.S. Department of Energy
through the LANL/LDRD Program.}}

\author[First]{Soumya Kundu} 
\author[Second]{Marian Anghel} 

\address[First]{Center for Nonlinear Studies and Information Sciences Group (CCS-3), Los Alamos National Laboratory, 
   Los Alamos, NM 87544 USA (e-mail: soumya@lanl.gov)}
\address[Second]{Information Sciences Group (CCS-3), Los Alamos National Laboratory, 
   Los Alamos, NM 87544 USA (e-mail: manghel@lanl.gov)}

\begin{abstract}                % Abstract of not more than 250 words.

We provide a distributed coordinated approach to the  stability analysis and control design of large-scale nonlinear dynamical systems by
using a vector Lyapunov functions approach.  In this formulation the large-scale system is decomposed into a network of interacting subsystems and the stability of the system is analyzed through a comparison system.  However finding such comparison system is not trivial. 
In this work, we propose a sum-of-squares based completely decentralized approach for computing the comparison systems for networks of nonlinear systems. Moreover, based on the comparison systems, we introduce a distributed optimal control strategy in which the individual 
subsystems (agents) coordinate with their immediate neighbors to design local control policies that can exponentially stabilize the full system under initial disturbances. 
%While this strategy is sufficient to stabilize the system, we show that the introduction of 
%an  upper level coordination agent (centralized controller), that uses minimal information from the lower level control agents, 
%can further improve the 
%performance by minimizing the control effort. 
We illustrate the control algorithm on a network of interacting Van der Pol systems. 

\end{abstract}


\end{frontmatter}
%===============================================================================

%% There are a number of predefined theorem-like environments in
%% ifacconf.cls:
%%
%% \begin{thm} ... \end{thm}            % Theorem
%% \begin{lem} ... \end{lem}            % Lemma
%% \begin{claim} ... \end{claim}        % Claim
%% \begin{conj} ... \end{conj}          % Conjecture
%% \begin{cor} ... \end{cor}            % Corollary
%% \begin{fact} ... \end{fact}          % Fact
%% \begin{hypo} ... \end{hypo}          % Hypothesis
%% \begin{prop} ... \end{prop}          % Proposition
%% \begin{crit} ... \end{crit}          % Criterion

%% NOTE: refer to tables as "table (xx)" [\begin{table}[hb]].

\section{Introduction}

Distributed coordinated control has  recently provided powerful control solutions
 when the conventional centralized methods fail due to inevitable communication constraints and limited computational capabilities. Paradigmatic examples are provided by  cooperative and coordinated control for
autonomous multi-agent systems (see \cite{Bullo:2009}) or large scale interconnected systems 
(see \cite{Siljak:2010}).
 Distributed coordinated control uses {\em local} communications between agents to achieve {\em global} objectives that reflect the desired behavior of the multi-agent system.
   Usually, a two-level hierarchical multi-agent system is employed, which consists of
upper level agent for implementing coordinated control and lower level
agents for implementing decentralized control.
In this paper, we propose to use this  conceptual framework to design
distributed coordinated control of
large scale interconnected system using vector Lyapunov functions (see \cite{Bellman:1962, Bailey:1966}) 
and comparison principles (see \cite{Brauer:1961, Beckenbach:1961}). 
The formulations using vector Lyapunov
functions are computationally very attractive because of their
parallel structure and scalability.  However computing these comparison equations, for a given
interconnected system, still remained a challenge.
In this work we use  sum-of-squares (SOS) methods to study the stability of an interconnected
system by computing  the vector Lyapunov functions as well as the comparison equations.
While this approach is applicable to any generic dynamical
system, we choose a randomly generated network of
modified\footnote{We choose the Van der Pol `oscillator' parameters in such a way that
these have a stable equilibrium at origin.} Van der Pol oscillators for illustration. 
%Each Van der Pol oscillator can be represented as a two-state system
%with state dynamic equations as polynomials of degree three . 
This network is decomposed into many interacting
subsystems and each subsystem parameters are chosen so that
individually each subsystem is stable, when the disturbances
from neighbors are zero. SOS based expanding interior
algorithm (see \cite{Wloszek:2003, Anghel:2013}) is used to obtain estimate of region of
attraction as sub-level sets of polynomial Lyapunov functions
for each such subsystem. Finally SOS optimization is used to
compute the stabilizing control policies, based on linear comparison systems, such that the closed-loop network is exponentially stable under initial disturbances. 

Following some brief background
in Section\,\ref{S:background} we formulate the control design problem in Section\,\ref{S:problem}.
The sum-of-squares based distributed control algorithm is proposed in Section\,\ref{S:control}. In Section\,\ref{S:results} we illustrate the control design on a network of
Van der Pol systems, before concluding the article in Section\,\ref{S:concl}.

%%%%%%%%%%%%%%%%%%%%%%%%%%%%%%%%%%%%%%%%%%%%%%%%%%%%%%%%%%%%%%%%%%%%%%%%%%%%%%%%%%%%%%%%%%%%%%%%%%%%%%%%
%===============================
\section{Preliminaries}
\label{S:background}
%===============================
%Let us briefly review some of the key concepts behind our analysis. We will first discuss how the stability of a dynamical system can be analyzed by constructing suitable Lyapunov functions. Then we briefly refer to sum-of-square polynomials and a very useful result which helps us in formulating the sum-of-squares problems. Finally we mention a nice result on linear ordinary differential equations which would 
%help connect the stability problem to the solution of a linear comparison equation.

%=====================================
%\subsection{Lyapunov Stability Methods}
\subsection{Stability and Control of Nonlinear Systems}
\label{S:Lyap}
%=====================================
Let us consider the dynamical systems of the form
\begin{align}\label{E:f}
&\dot{x}\left(t\right) = f\left(x\left(t\right)\right)+u_t\,,\quad t\geq 0\,,~~ f(0)=0\,,
\end{align}
where $x\!\in\!\mathbb{R}^n$ are the states, $u_t\!\in\!\mathbb{R}^n$ are the control input, $f:\mathbb{R}^n\rightarrow \mathbb{R}^n$ is locally Lipschitz and the origin is an equilibrium point\footnote{State variables can be shifted to move any equilibrium point to the origin.} of the `free' system, i.e. the system with no control ($u_t\equiv 0$). 
Let us first review the important concepts on stability of the equilibrium point of the `free' system.
\begin{definition}\label{D:stability}
The equilibrium point at the origin is called asymptotically stable in a domain $\mathcal{D}\!\subseteq\!\mathbb{R}^n,\,0\!\in\!\mathcal{D},$ if 
\begin{align*}
\left\|x(0)\!\right\|_2\ \!\!\!\!\in\!\! \mathcal{D}\implies\lim_{t\rightarrow\infty}\left\|x(t)\right\|_2\!=\!0\,,
\end{align*}
and it is exponentially stable if there exists $b,c \!>\! 0$ such that
\begin{align*}
\left\|x(0)\!\right\|_2\ \!\!\!\!\in\!\! \mathcal{D} \!\!\!\implies\!\!\!\! \left\|x(t)\right\|_2\!<\!ce^{-bt}\!\left\| x(0)\right\|_2~\forall t\!\geq\! 0\,.
\end{align*}
\end{definition}
Lyapunov's first or direct method (see \cite{Lyapunov:1892,Slotine:1991}) can give a sufficient condition of stability through the construction a certain positive definite function.

\begin{thm}\label{T:Lyap}
If there exists a domain $\mathcal{D}\!\!\in\!\!\mathbb{R}^n$, $0\!\!\in\!\!\mathcal{D}$, and a continuously differentiable positive definite function {$\tilde{V}\!\!:\!\!\mathbb{R}^n\!\rightarrow\! \mathbb{R}$}, called the `Lyapunov function' (LF), then the equilibrium point of the `free' system at the origin is asymptotically stable if $\nabla{\tilde{V}}^T\!\!f(x)$ is negative definite in $\mathcal{D}$, and is exponentially stable if $\nabla{\tilde{V}}^T\!\!f(x)\leq\!-c\, \tilde{V}~\forall x\!\in\!\mathcal{D}$, for some $c>0$.
\end{thm}
When there exists such a $\tilde{V}\left(x\right)$, the region of attraction (ROA) of the equilibrium point at the origin can be estimated as
\begin{subequations}\label{E:ROA}
\begin{align}
\mathcal{R}:=&~\left\lbrace x\in\mathcal{D}\left|\, {V}(x)\leq 1\right.\right\rbrace,\\
\text{where,}~{V}(x) =&~{\tilde{V}(x)}/{\gamma^{max}},~\text{and}\\
	\gamma^{max}:=&~\arg\max_\gamma~\left\lbrace x\in\mathbb{R}^n\left| \tilde{V}(x)\leq\gamma\right.\right\rbrace \subseteq \mathcal{D}.
\end{align}
\end{subequations}
%Henceforth, without any loss of generality, we assume that the ROA of the equilibrium at origin of the `free' system is estimated as sub-unity-level set of the properly scaled LF.

For systems under some control action $u_t$, the notion of `stabilizability' becomes important. Specifically, we are interested in state-feedback control of the form $u_t = u_t\left(x\right)$.

\begin{definition}\label{D:stabilizable}
The system \eqref{E:f} is called (exponentially) stabilizable if there exists a control policy $u_t\!=\!u_t\left(x\right)\!,\,\,t\geq 0$, such that the origin of the closed-loop system is (exponentially) stable, in which case $u_t$ is called a (exponentially) stabilizing control.
\end{definition}

Courtesy to the works of \cite{Artstein:1983} and \cite{Sontag:1989}, the concept of `control Lyapunov functions' has been useful in the context of stabilizability. 

\begin{definition}\label{D:CLF}
A continuously differentiable positive definite function $V_c:\mathbb{R}^n\rightarrow\mathbb{R}$ is called a `control Lyapunov function' (CLF) if for each $x\in\mathbb{R}^n\backslash\lbrace 0\rbrace$, there exists a control $u_t$ such that $\nabla V_c^T\!\!\left(f(x)+u_t\right)<0$.
\end{definition}

Similar definition holds for `exponentially stabilizing' CLFs (see \cite{Ames:2014, Zhang:2009}). CLFs can easily accommodate `optimality' in the control policies as well (see \cite{Freeman:2008}). However, as with the LFs, it is often very difficult to find a CLF for a given system.

%==============================================
\subsection{Sum-of-Squares and Positivstellensatz Theorem}
\label{subsec:SOSmethod}
%==============================================
In recent years, sum-of-squares (SOS) based optimization techniques have been successfully used in constructing LFs by restricting the search space to sum-of-squares polynomials (see \cite{Wloszek:2003,Parrilo:2000,Tan:2006,Anghel:2013}). Let us denote by $\mathbb{R}\left[x\right]$ the ring of all polynomials in $x\in\mathbb{R}^n$. Then,
{\begin{definition}
A multivariate polynomial {$p \in \mathbb{R}[x]$, $x\in\mathbb{R}^n$}, is called a sum-of-squares (SOS) if there exists $h_i\in\mathbb{R}[x]$, $i\in\left\lbrace 1,\dots,s\right\rbrace$, for some finite $s$, such that $p(x) = \sum_{i=1}^s h_i^2(x)$. 
Further, the ring of all such SOS polynomials is denoted by $\Sigma[x]$.
%Further, we denote the ring of all SOS polynomials in $x\in\mathbb{R}^n$ by $\Sigma[x]$.
\end{definition}
Checking if $p\in\mathbb{R}[x]$ is an SOS is a semi-definite problem} which can be solved with a MATLAB$^\text{\textregistered}$ toolbox SOSTOOLS (see \cite{sostools13,Antonis:2005b}) along with a semidefinite programming solver such as SeDuMi (see \cite{Sturm:1999}).
{SOS technique can be used to search for polynomial LFs, by translating the conditions in Theorem\,\ref{T:Lyap} to equivalent SOS conditions (see \cite{Wloszek:2003,Wloszek:2005,Antonis:2005a}).}
An important result from algebraic geometry called Putinar's Positivstellensatz theorem\footnote{Refer to \cite{Lasserre:2009} for other versions of the Positivstellensatz theorem.} (see \cite{Putinar:1993,Lasserre:2009}) helps in translating the SOS conditions into SOS feasibility problems. 
\begin{thm}\label{T:Putinar}
Let $\mathcal{K}\!\!=\! \left\lbrace x\in\mathbb{R}^n\left\vert\, k_1(x) \geq 0\,, \dots , k_m(x)\geq 0\!\right.\right\rbrace$ be a compact set, where $k_j\!\in\!\mathbb{R}[x]$, $\forall j\in\left\lbrace 1,\dots,m\right\rbrace$. Suppose there exists a $\mu\!\in\! \left\lbrace \sigma_0 + {\sum}_{j=1}^m\sigma_j\,k_j \left\vert\, \sigma_0,\sigma_j \in \Sigma[x]\,,\forall j \right. \right\rbrace$ such that $\left\lbrace \left. x\in\mathbb{R}^n \right\vert\, \mu(x)\geq 0 \right\rbrace$ is compact. Then, if $p(x)\!>\!0~\forall x\!\in\!\!\mathcal{K}$, then $p \!\in\! \left\lbrace \sigma_0 \!\!+\!\! \sum_j\sigma_jk_j\!\!\left\vert\, \sigma_0,\sigma_j\!\!\in\!\Sigma[x],\forall j\!\right.\right\rbrace$.
\end{thm}
In many cases, especially for the $k_i\,\forall i$ used throughout this work, a $\mu$ satisfying the conditions in Theorem\,\ref{T:Putinar} is guaranteed to exist (see \cite{Lasserre:2009}), and need not be searched for.

%==============================================
\subsection{Linear Comparison Principle}
\label{subsec:comparison}
%==============================================
Before finishing this section, let us take a look at a nice result on the ordinary differential equations which helps form the framework of stability analysis of inter-connected systems via vector LFs. Noting that all the elements of the vector $e^{At},~ t\geq 0$, where $A=\left[a_{ij}\right]\in\mathbb{R}^{m\times m}$, are non-negative if and only if $a_{ij}\geq 0, i\neq j$, the authors in \cite{Beckenbach:1961,Bellman:1962} proposed the following result:
\begin{lemma}\label{L:comparison}
Let $A=[a_{ij}]\in\mathbb{R}^{m\times m}$ have only non-negative off-diagonal elements, i.e. $a_{ij}\geq 0,~i\neq j$. Then 
\begin{align}\label{E:comp_ineq}
\dot{v}(t)\leq A\,v(t), ~t\geq 0, ~v\in\mathbb{R}^n, ~v(0) = v_0, 
\end{align}
implies $v(t)\leq r(t),~\forall t\geq 0$, where 
\begin{align}\label{E:comp_eq}
\dot{r}(t)= A\,r(t), ~t\geq 0, ~r\in\mathbb{R}^n, ~r(0) = v(0) = v_0. 
\end{align}
\end{lemma}
This result will henceforth be referred to as the `linear comparison principle' and the differential equation in \eqref{E:comp_eq} as the `comparison equation'.

%%%%%%%%%%%%%%%%%%%%%%%%%%%%%%%%%%%%%%%%%%%%%%%%%%%%%%%%%%%%%%%%%%%%%%%%%%%%%%%%%%%%%%%%%%%%%%%%%%%%%%%
%=========================================
\section{PROBLEM DESCRIPTION}\label{S:problem}
%=========================================
The problem of interest for this work is to find state-feedback control $u_t = u_t\left(x\right)$ that exponentially stabilizes a large nonlinear system \eqref{E:f}. One approach could be to find a suitable CLF (Definition\,\ref{D:CLF}), using computational methods, e.g. SOS technique. However, as noted in \cite{Antonis:2012}, such an approach will quickly become intractable as the system size increases. Instead, we seek distributed stabilizing control policies by modeling the large dynamical system as an interconnected network of $m$ ($\geq 2$) interacting subsystems,
\begin{subequations}\label{E:fi}
\begin{align}
\forall i =1,2,&\dots,m,\nonumber\\
\mathcal{S}_i:~&\dot{x}_i = f_i(x_i) + u_{t,i} + g_i(x), ~ x_i\!\in\!\mathbb{R}^{n_i}, ~x\!\in\!\mathbb{R}^n\\
&f_i({0})={0},\\
&g_i(\hat{x}_{i})={0},~\forall \hat{x}_{i}\in\left\lbrace x\in\mathbb{R}^n\!\left\vert ~x_j \!=\! 0,\forall j\!\neq\! i\right.\right\rbrace\\
%&g_i(x_{0,i})={0},~\forall x_{0,i}\in\left\lbrace x\in\mathbb{R}^n\left\vert x_j = 0~\forall j\neq i\right.\right\rbrace\\
\text{where,}~ x &= {\bigcup}_{j=1}^m\left\lbrace x_j\right\rbrace , ~\text{and}~n\leq{\sum}_{j=1}^m n_j\,.
\end{align}
\end{subequations}

We assume that the isolated `free' subsystem dynamics $f_i\in\mathbb{R}[x_i]^{n_i}$, and the neighbor interactions $g_i\in\mathbb{R}[x]^{n_i}$ are vectors of polynomials. Further, $u_{t,i}=u_{t,i}\left(x_i\right)$ is a time-dependent local state-feedback control policy, with each $u_{t,i}\in\mathbb{R}[x_i]^{n_i}\,\,\forall t$. It is assumed that the `free' isolated subsystems as well as the `free' full system are (locally) stable. Note that, we allow over-lapping decomposition in which subsystems can have common state(s) \cite{Siljak:1978,Jocic:1977}. Let 
\begin{subequations}\label{E:Ni}
\begin{align}
\mathcal{N}_i &:= \left\lbrace i\right\rbrace\cup\left\lbrace j\left\vert\begin{array}{l} g_i(x)\neq 0 ~\text{for some $x$}\\\text{with $x_k=0\,\forall k\neq i, j$}\end{array}  \right.\right\rbrace, \\
\text{and }~\bar{x_{i}} &:={\bigcup}_{j\in\mathcal{N}_i}\,\left\lbrace x_j\right\rbrace
\end{align}\end{subequations}
denote the set of indices of the subsystems in the neighborhood of $\mathcal{S}_i$ (including the subsystem itself) and the states that belong to this neighborhood, respectively.

The goal is to compute the distributed control $u_{t,i}(x_i)\,\forall i$ so that the full interconnected system \eqref{E:fi} is exponentially stabilizable.

%===============================================
\subsection{Comparison Equations and Exponential Stabiltiy}
\label{subsec:inter_stab}
%===============================================
Let us first review the stability of the `free' interconnected system, i.e. when $u_{t,i}\equiv 0\,\forall i$. Stability of each of the `free' isolated (i.e. zero neighbor interaction) subsystems
\begin{align}\label{E:fi_isol}
\forall  i \in\left\lbrace 1,2,\dots,m\right\rbrace,\quad &\dot{x}_i = f_i(x_i),~x_i\in\mathbb{R}^{n_i}  \, .
\end{align}
can be characterized by computing a polynomial LF $V_i (x_i)\,\forall i\,$, and the corresponding estimate of the ROA as in \eqref{E:ROA}.
An SOS based \textit{expanding interior algorithm}, (see \cite{Wloszek:2003,Anghel:2013}), is used to iteratively enlarge the estimate of the ROA by finding a `better' LF at each step of the algorithm. At the completion of this iterative algorithm, the stability of each `free' isolated subsystem \eqref{E:fi_isol} is quantified by its LF $V_i(x_i)$, with a corresponding estimate of the ROA as
 \begin{align}
\mathcal{R}_i^0:= \left\lbrace x_i\in\mathbb{R}^{n_i}\left| V_i(x_i)\leq 1\right.\right\rbrace,~\forall i=1,2,\dots,m\,.
 \end{align}
Let us further define the domain
\begin{align}\label{E:ROA_isol}
\mathcal{R}^{0} &:=\left\lbrace x\in\mathbb{R}^{n}\left| ~x_i\in\mathcal{R}_i^0,\,~ \forall i=1,2,\dots,m\right.\right\rbrace .
\end{align}
%which could be interpreted as the ROA of the `free' interconnected system \eqref{E:fi}, in absence of the all the interactions. The disturbances coming from the neighbors can be expressed by the subsystem LF level-sets. 
The equilibrium of the `free' network at the origin corresponds to the zero level-sets, $V_i(0)=0\,\,\forall i\,$, and any initial condition away from this equilibrium would result in positive level-sets $V_i(x_i(0))\!=\!\gamma_i^0\!\in\!\left(0,1\right]$ for some or all of the subsystems. 

%A necessary and sufficient condition of asymptotic stability (Definition~\ref{D:stability}) can then be translated into the condition
%\begin{align}\label{E:cond_asymp}
%\forall i, ~V_i(x_i(0))=\gamma_i^0\implies\forall i, ~\lim_{t\rightarrow +\infty}{V}_i(x_i(t))=0\, ,
%\end{align}
%where $x_i(t),~t>0$, are solutions of the coupled dynamics in \eqref{E:fi}. Even though \eqref{E:cond_asymp} reduces the dimensionality of the problem, it still remains a generally non-trivial problem.
An attractive and scalable approach for (exponential) stability analysis of the `free' network uses a vector LF (see \cite{Bellman:1962,Bailey:1966})
\begin{align}\label{E:vecLyap}
V(x) &:= \left[V_1(x_1)  ~~ V_2(x_2) ~~\dots ~~V_m(x_m)\right]^T
\end{align} 
to construct a linear comparison equation (Lemma\,\ref{L:comparison}) whose states are the subsystem LFs (see \cite{Siljak:1972,Weissenberger:1973,Araki:1978}). The aim is to seek an $A=[a_{ij}]\in\mathbb{R}^{m\times m}$ and a domain $\mathcal{D}\subset\mathcal{R}^0$, such that
\begin{subequations}\label{E:comparison}
\begin{align}
\left.\dot{V}(x)\right\vert&_{u_{t,i}\equiv 0\,\forall i}\,\leq~ AV(x),~\forall x\in\mathcal{D}\subset\mathcal{R}^0, \label{E:comparison_VAV}\\
\text{where,}\quad & a_{ij}\geq 0~\forall i\neq j\, ,\label{E:Metzler}\\
		 &  \text{$A=[a_{ij}]$ is Hurwitz, and} \label{E:Hurwitz}\\
		\quad & \text{$\mathcal{D}$ is invariant under the dynamics \eqref{E:f}\,,}\label{E:invariance}\\
\text{and} ~& \left.\dot{V}(x)\right\vert_{u_{t,i}\equiv 0\,\forall i} = \left[ \begin{array}{c}\nabla{V}_1^T\!\left(f_1(x_1)+g_1(x)\right)\\
%						\nabla{V}_2^T\!\left(f_2(x_2)+g_2(x)\right)\\
						\vdots \\
						\nabla{V}_m^T\!\left(f_m(x_m)+g_m(x)\right) \end{array}\right]. \label{E:Vdot}
\end{align}
\end{subequations}
If there exist a `comparison matrix' $A=[a_{ij}]$ and $\mathcal{D}\subset\mathcal{R}^0$ satisfying \eqref{E:comparison}, then any $x(0)\in\mathcal{D}$ would guarantee exponential convergence of $V(x(t))$ to the origin thereby implying exponential convergence of the states themselves (see \cite{Siljak:1972}).

%===============================================
\subsection{Exponentially Stabilizing Control}
%\label{subsec:inter_stab}
%===============================================
The comparison principle can be used to design distributed controllers $u_{t,i}(x_i)\,\forall i$ that exponentially stabilize the nonlinear network \eqref{E:fi}. In Section\,\ref{S:control}, we propose an SOS based algorithmic approach in which each of the subsystems $\mathcal{S}_i$ coordinates only with its immediate neighbors $\mathcal{S}_j\,,\,j\!\in\!\!\mathcal{N}_i\backslash\lbrace i\rbrace$, to compute a local and `optimal' stabilizing control $u_{t,i}\,$.

We propose that the LFs for each `free' (no control) and isolated (no interaction) subsystem \eqref{E:fi_isol} be pre-computed and communicated to the neighbors. Given any initial condition $x(0)\!\in\!\mathcal{R}^0$ we define the domain
\begin{align}\label{E:D}
\mathcal{D}:= \left\lbrace x\in\mathcal{R}^0\left\vert\, V_i(x_i)\leq V_i(x_i(0))\!=\!\gamma_i^0~~\forall i\right.\right\rbrace.
\end{align}
Then any distributed control $u_{t,i}(x_i)\,\forall i$ satisfying
\begin{subequations}\label{E:comparison_control}
\begin{align}
\dot{V}(x)\, &\leq~ AV(x),~\forall x\in\mathcal{D}\subset\mathcal{R}^0, \label{E:comparison_control_VAV}\\
\text{s.t.,}~~ & \text{conditions \eqref{E:Metzler}, \eqref{E:Hurwitz} and \eqref{E:invariance}\,,}\\
\text{where}\, ~& \dot{V}(x) = \left[ \begin{array}{c}\nabla{V}_1^T\!\left(f_1(x_1) + u_{t,1}(x_1)+g_1(x)\right)\\
%						\nabla{V}_2^T\!\left(f_2(x_2)+g_2(x)\right)\\
						\vdots \\
						\nabla{V}_m^T\!\left(f_m(x_m) + u_{t,m}(x_m)+g_m(x)\right) \end{array}\right]. \label{E:Vdot_control}
\end{align}
\end{subequations}
is an exponentially stabilizing control policy. In addition to satisfying \eqref{E:comparison_control}, the `optimality' of the control could be ascertained by minimizing the applied control efforts.

\begin{remark}
Note that we do not explicitly compute a CLF (Definition\,\ref{D:CLF}), because of the computational burden in large-scale networks. Instead, we propose an algorithm to design stabilizing control using the pre-computed subsystem LFs.
\end{remark}

%===============================================
\section{Distributed Control Algorithm}
\label{S:control}
%===============================================
In designing the stabilizing control policies $u_{t,i}\,\forall i$ in \eqref{E:comparison_control} the conditions \eqref{E:Hurwitz} and \eqref{E:invariance} have to be satisfied, which essentially demands availability of network-level information. However, the following two key observation can be useful in generating equivalent subsystem-level conditions. 
\begin{proposition}\label{P:Gershgorin}
A matrix $A=[a_{ij}]\in\mathbb{R}^{m\times m}$ is Hurwitz if, for each $i\in\lbrace 1,2,\dots,m\rbrace$, $a_{ii}+\sum_{j\neq i}\left\vert a_{ij}\right\vert <0$\,.\footnote{In other words, a strictly diagonally-dominant matrix with negative diagonal entries is Hurwitz.}
\end{proposition}
\begin{proof}
From the Gershgorin's Circle theorem (see in \cite{Bell:1965,Gershgorin:1931}), for every eigenvalue $\lambda\in\mathbb{C}$ of the matrix $A=[a_{ij}]$, 
\begin{align*}
%\label{E:gershgorin}
\exists\, k\in\left\lbrace 1,2,\dots,m\right\rbrace ~\text{such that,}~\left|\lambda - a_{kk}\right|&\leq {\sum}_{j\neq k}\left|a_{kj}\right|\,.
\end{align*}
Using $\sum_{j\neq k}\left|a_{kj}\right|<-a_{kk}$, it follows that $\text{Re}\lbrace \lambda\rbrace< 0$.\hfill\hfill\qed
\end{proof}
Additionally, we also note that (see \cite{Weissenberger:1973}),
\begin{proposition}\label{P:invariance}
The domain $\mathcal{D}$ in \eqref{E:D} is invariant if $\sum_{j=1}^m a_{ij}\,\gamma_i^o\!\leq\!0$, where $A=[a_{ij}]$ satisfies the comparison equation \eqref{E:comparison_control_VAV}.
\end{proposition}
\begin{proof}
We note that whenever $V_i(x_i(\tau))=\gamma_i^0$, for some $i\,$, and $V_k(x_k(\tau))\leq\gamma_k^0~\forall k\!\neq\!i\,$, for some $\tau\geq 0$, we have
\begin{align*}
%\label{E:proof1}
\left.\dot{V}_i\left(x_i\right)\right\vert_{t=\tau} &\leq a_{ii}\gamma_i^0 + {\sum}_{k\neq i} a_{ik}V_k\left(x_k(\tau)\right) \leq 0\,.
\end{align*}
i.e. the (piecewise continuous) trajectories can never cross the boundaries defined as $\left\lbrace x\in\mathcal{D} \left\vert ~V_i\left(x_i\right)=\gamma_i^0~\forall i\right.\right\rbrace\,$. \hfill\hfill\qed
\end{proof}

Propositions\,\ref{P:Gershgorin} and \ref{P:invariance} can be used to replace the network-level conditions \eqref{E:Hurwitz} and \eqref{E:invariance}, respectively, by their equivalent decentralized, albeit more conservative, conditions to facilitate design of distributed control policies $u_{t,i}\,\forall i\,$ that satisfy
\begin{subequations}\label{E:control}
\begin{align}
\forall i: &~\nabla V_i^T\!\!\!\left(f_i(x_i) \!+\! u_{t,i}(x_i) \!+\! g_i(x)\!\right) \!\leq\! \!\sum_{j\in\mathcal{N}_i}\!\!a_{ij} V_j(x_j\!)~~\forall x\!\in\!\mathcal{D}, \\
	&\text{subject to:}~\left\lbrace\begin{array}{l} a_{ij}\geq 0~~\forall j\!\in\!\!\mathcal{N}_i\backslash\lbrace i\rbrace\,,\\ 									{\sum}_{j\!\in\!\mathcal{N}_i}\,a_{ij}<0\,,~\text{and}\\
								{\sum}_{j\!\in\!\mathcal{N}_i}\,a_{ij}\,\gamma_j^0\leq 0\,.\end{array}\right.
\end{align}\end{subequations}
Note that, $a_{ij}\!=\!0~\forall j\!\notin\!\!\mathcal{N}_i\,$. Using the Positivstellensatz theorem (Theorem~\ref{T:Putinar}), with $k_i=\!\left(\gamma_i^0\!\!-\!V_i(x_i)\right)~\forall i\,$, and $\mathcal{K}\!\!=\!\mathcal{D}$, we can cast \eqref{E:control} into a set of SOS feasibility problems, for each $i\,$,
\begin{subequations}\label{E:control_SOS}
\begin{align}
-\!\nabla V_i^T\!\!\left(f_i\!+ u_{t,i} +\! g_i\right) + \!\!\!\sum_{j\in\mathcal{N}_i} \!\!\left(a_{ij}V_j \!- \sigma_{ij}\!\left(\gamma_j^0\!-\!V_j \right)\!\right) \!\in \Sigma[\bar{x}_i], & \\
 -{\sum}_{j\!\in\!\mathcal{N}_i}\,a_{ij}\in\Sigma[0]\,, &\\
 ~\text{and}~ \,-{\sum}_{j\!\in\!\mathcal{N}_i}\,a_{ij}\,\gamma_j^0\in\Sigma[0]\,, &\\
~\text{where}~\,u_{t,i}\in\mathbb{R}[x_i]^{n_i},~\sigma_{ij}\!\in\!\Sigma[\bar{x}_i]~\forall \!j\!\in\!\mathcal{N}_i\,,&\\
a_{ii}\!\in\!\mathbb{R}[0]\,,~\text{and}~\, a_{ij}\!\in\!\Sigma[0]~\forall\! j\!\in\!\mathcal{N}_i\backslash\lbrace i\rbrace\,. &
\end{align}\end{subequations}
Here $\mathbb{R}[0]$ denotes scalar variables, $\Sigma[0]$ denotes non-negative scalar variables and $\bar{x}_i$ were defined in \eqref{E:Ni}.

The set of SOS conditions \eqref{E:control_SOS} defines the control $u_{t,i}\in\mathbb{R}[x_i]^{n_i}$ as an $n_i$-vector of polynomials in $x_i$, of a chosen degree. But further restrictions can be imposed on the control design. In this work, we consider bounded control signals of the form
\begin{subequations}\label{E:u_bound}
\begin{align}
\forall i: \quad &\left\vert\, u_{t,i,k}(x_i)\right\vert \leq \bar{U}_{i,k}\quad\forall t\geq 0\,,~\forall k\!\in\!\lbrace 1,2,\dots,n_i\rbrace\\\
		\text{where},~ & \left\lbrace\begin{array}{l} u_{t,i} = [\,u_{t,i,1}\,~u_{t,i,2}\,~\dots~\,u_{t,i,n_i}\,]^T,\\
			\bar{U}_{i,k} \geq 0 \quad \forall k\!\in\!\lbrace 1,2,\dots,n_i\rbrace \,. \end{array}\right.
\end{align}\end{subequations}
%\begin{subequations}\label{E:u_bound}
%\begin{align}
%\forall i: \qquad &\left\vert\, u_{t,i}\right\vert \leq \bar{U}_i\quad\forall t\geq 0\,,~\\
%		\text{where},~ & \left\lbrace\begin{array}{l}\bar{U}_i = [\,\bar{U}_{i,1}\,~\bar{U}_{i,2}\,~\dots~\,\bar{U}_{i,n_i}\,]^T\in\mathbb{R}^{n_i}\\
%			\bar{U}_{i,k} \geq 0 \quad \forall k\!\in\!\lbrace 1,2,\dots,n_i\rbrace \end{array}\right.
%\end{align}\end{subequations}
For the uncontrolled states, we set the corresponding control bounds to zero. Further, by declaring these bounds as design variables the control problem can be formulated as a minimization of the maximal control efforts as,
\begin{subequations}\label{E:control_optimal}
\begin{align}
\forall i: ~& \underset{u_{t,i,k}\,,~a_{ij}}{\text{minimize}} ~~{\sum}_{k=1}^{n_i}\bar{U}_{i,k} \\
	\text{s.t.,}~~ &\text{conditions \eqref{E:control_SOS}, }\\
	&			\bar{U}_{i,k}\!-u_{t,i,k} -\! {\sigma}_{i,k}^{up}\! \left(\gamma_i^0 \!-\! V_i\right) \in \Sigma[x_i]\,,~\forall k\!\in\!\lbrace 1,\!\dots\!, n_i\rbrace,\\
	& 	\bar{U}_{i,k}\!+u_{t,i,k} -\! {\sigma}_{i,k}^{low}\! \left(\gamma_i^0 \!-\! V_i\!\right) \in \Sigma[x_i]\,,~\forall k\!\in\!\!\lbrace 1,\!\dots\!, n_i\rbrace, \\
	& \text{where,}~\, \bar{U}_{i,k}~\left\lbrace \begin{array}{ll}\! =0\,, \!&\text{for the uncontrolled states, }\\
							 	\!\in\Sigma[0]\,, \!&\text{for the controlled states.}\end{array}\!\right. \\
	&\text{and}~\sigma_{i,k}^{up/low}\!\in\Sigma[x_i]~\forall k\in\lbrace 1,2,\dots,n_i\rbrace\,.
							 \end{align}\end{subequations}
Given a choice of the degree of the control polynomials and an initial condition, \eqref{E:control_optimal} can be solved to find optimal, distributed, and exponentially stabilizing control policies. Algorithm\,\ref{ALG:control} outlines the major steps in the proposed control design procedure. 

It should be noted that for the subsystems that do not need to apply control the solution of the optimization \eqref{E:control_optimal} would result in $\bar{U}_{i,k}=0\,\forall k\!\in\!\lbrace 1,\!\dots\!,n_i\rbrace$. 

\begin{algorithm}
\caption{Distributed Stabilizing Control Design}
\label{ALG:control}
\begin{algorithmic}[]
\Procedure{One-time Computation}{}
\For{each subsystem $i\in\left\lbrace 1,2,\dots,m\right\rbrace$}
\State {Compute the LF $V_i(x_i)$ based on \eqref{E:fi_isol}}
\State {Communicate $V_i$ to neighbors $\mathcal{S}_j~\forall \!j\!\in\!\mathcal{N}_i\backslash\lbrace i\rbrace$}
\State {Receive and store the LFs $V_j(x_j)~\forall\! j\!\in\!\mathcal{N}_i$}
\EndFor
\EndProcedure
\end{algorithmic}

\begin{algorithmic}[]
\Procedure{Real-time Computation}{}
\For{each subsystem $i\in\left\lbrace 1,2,\dots,m\right\rbrace$}
\State {Compute initial level-set $\gamma_i^0 = V_i(x_i(0))$}
\State {Communicate $\gamma_i^0$ to neighbors $\mathcal{S}_j~\forall \!j\!\in\!\mathcal{N}_i\backslash\lbrace i\rbrace$}
\State {Receive $\gamma_j^0~\forall \!j\!\in\!\mathcal{N}_i\backslash\lbrace i\rbrace$ from neighbors }
\State {Solve \eqref{E:control_optimal} for the optimal control input $u_{t,i}(x_i)$ }
\EndFor
\EndProcedure
\end{algorithmic}
\end{algorithm}

\begin{remark}
Often in practical scenarios, the control bounds need to be strictly imposed due to physical considerations, in which case the degree of the control polynomials can be varied to find feasible control policies. 
\end{remark}

%===============================================
\section{Example}
\label{S:results}
%===============================================

%%==============================================
% \subsection{Model Description}\label{S:model}
% %==============================================
 {W}{e} consider a network of nine Van der Pol `oscillators' (see \cite{van:1926}), with parameters of each oscillator chosen to make them individually (exponentially) stable (without the control). Each Van der Pol oscillator is treated as an individual subsystem, with the interconnections as shown below,
\begin{align}
\begin{array}{lll}
\mathcal{N}_1:\left\lbrace 1, 2, 5, 9\right\rbrace & \mathcal{N}_2:\left\lbrace 2, 1, 3\right\rbrace & \mathcal{N}_3:\left\lbrace 3, 2, 8\right\rbrace\\ 
\mathcal{N}_4:\left\lbrace 4, 6, 7\right\rbrace & \mathcal{N}_5:\left\lbrace 5, 1, 6\right\rbrace & \mathcal{N}_6:\left\lbrace 6, 4, 5\right\rbrace\\
\mathcal{N}_7:\left\lbrace 7, 4, 8, 9\right\rbrace & \mathcal{N}_8:\left\lbrace 8, 3, 7\right\rbrace & \mathcal{N}_9:\left\lbrace 9, 1, 7\right\rbrace\,.
\end{array}
\end{align}
Each subsystem $\mathcal{S}_i~\forall i\!\in\!\lbrace 1,2,\dots,9\rbrace$ has two state variables, $x_i=\left[\,x_{i,1}~\,x_{i,2}\,\right]^T$. 
%where the first state, $x_{i,1}\,\forall i\,$, is considered to be uncontrolled. 
The subsystem dynamics, under the presence of the neighbor interactions and control input, is given by 
 \begin{subequations}
\begin{align*}
\mathcal{S}_i:~	&\dot{x}_{i,1}= x_{i,2}\,, \\
	&\dot{x}_{i,2}=\alpha_i\,x_{i,2}\!\left(1\!-\!x_{i,1}^2\!\right) \!-\! x_{i,1} \!+ u_{t,i,2}+x_{i,1}\!\!\!\!\!\!\sum_{k\in\mathcal{N}_i\backslash\left\lbrace i\right\rbrace}\!\!\!\!\!\!\beta_{ik}\,x_{k,2} \,.
\end{align*}
\end{subequations}
where the subsystem parameters $\alpha_i\!\in\![-\!2\,,-\!1]\,\,\forall i$ and the interaction parameters $\beta_{ik}\!\in\![-0.8\,,\,0.8]~\forall i\,,\forall k\!\in\!\mathcal{N}_i\backslash\left\lbrace i\right\rbrace$, are chosen randomly. Note that, we have considered $u_{t,i,1}\equiv 0~\forall t\,\forall i\,$, i.e. the state variables $x_{i,1}\,\forall i\,$ are not (directly) controlled. 

The goal is to apply the Algorithm\,\ref{ALG:control} to compute distributed optimal controllers $u_{t,i,2}(x_{i,1},x_{i,2})~\forall i$ that guarantee exponential stabilization of the network of Van der Pol systems.

%===================================================
 \subsection{Pre-Computation of Lyapunov Functions}
 %==================================================
At first, we compute polynomial Lyapunov functions for the isolated (interaction free) and control-free subsystems
 \begin{subequations}\label{E:}
\begin{align}
\forall i:~	&\dot{x}_{i,1}= x_{i,2}\,, \\
	&\dot{x}_{i,2}=\alpha_i\,x_{i,2}\left(1\!-\!x_{i,1}^2\!\right) \!-\! x_{i,1} \,,
\end{align}
\end{subequations}
using the \textit{expanding interior algorithm} (Section\,\ref{subsec:inter_stab}). As an example, we show a quadratic Lyapunov function and the associated estimate of the ROA of the interaction-free and control-free subsystem $\mathcal{S}_9$, 
\begin{subequations}\label{E:ROA_9}\begin{align}
\mathcal{R}_9^0&=\left\lbrace \left(x_{9,1},x_{9,2}\right)\left\vert~ V_9\leq 1\right.\right\rbrace, \\
\text{where,}~V_9&=0.595\,x_{9,1}^2 + 0.227\,x_{9,1}\,x_{9,2} + 0.520\,x_{9,2}^2\,.
\end{align}\end{subequations}
Fig.\,\ref{F:ROAcompare} shows a comparison of the estimated ROA using the quadratic LF in \eqref{E:ROA_9}, another estimate using a quartic LF and the `true' ROA computed numerically by simulating the isolated and free dynamics. Clearly, the estimate improves with higher order LFs. However, for computational ease, the rest of the analysis will be based on quadratic LFs. 

\begin{figure}[thpb]
      \centering
	\includegraphics[scale=0.5]{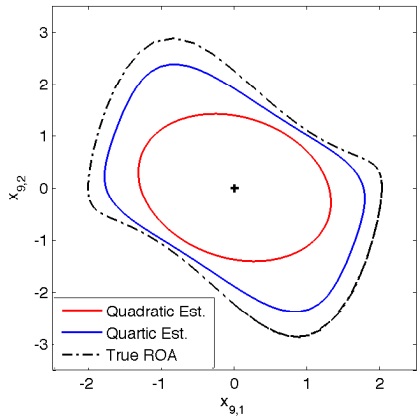}
      \caption{Estimated ROA for isolated and free subsystem $\mathcal{S}_9$.}
      \label{F:ROAcompare}
   \end{figure}   
   
Note that these LFs are computed only once for the network, and stored to be used for real-time control design. 
%with minimal coordinate between neighboring subsystems, as illustrated by the following test case. 

%===================================================
 \subsection{Controller Design: Test Case}
 %==================================================
 \begin{figure*}[thpb]
\centering
\subfigure[]{
\includegraphics[scale=0.5]{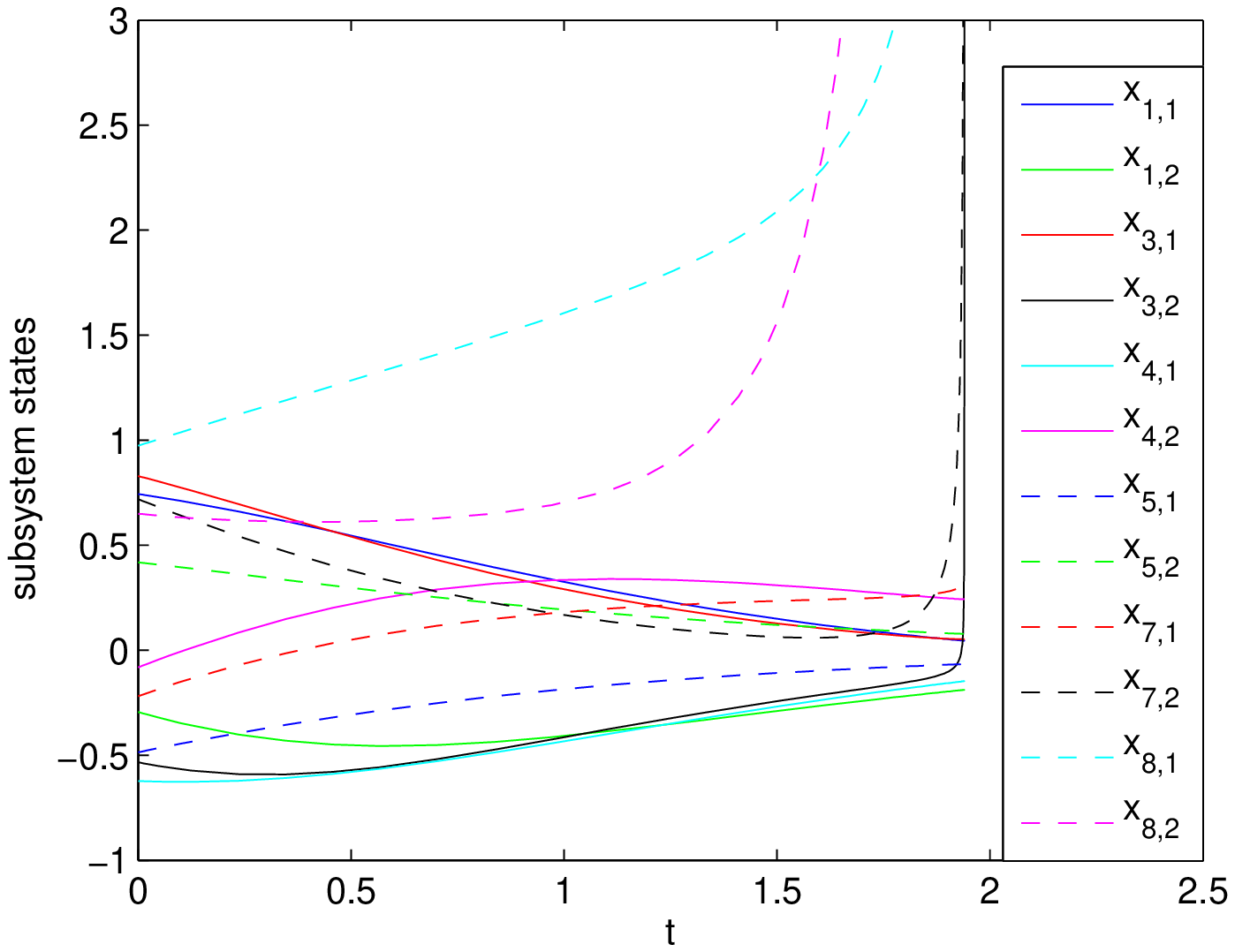}\label{F:states_42}
}
\subfigure[]{
\includegraphics[scale=0.5]{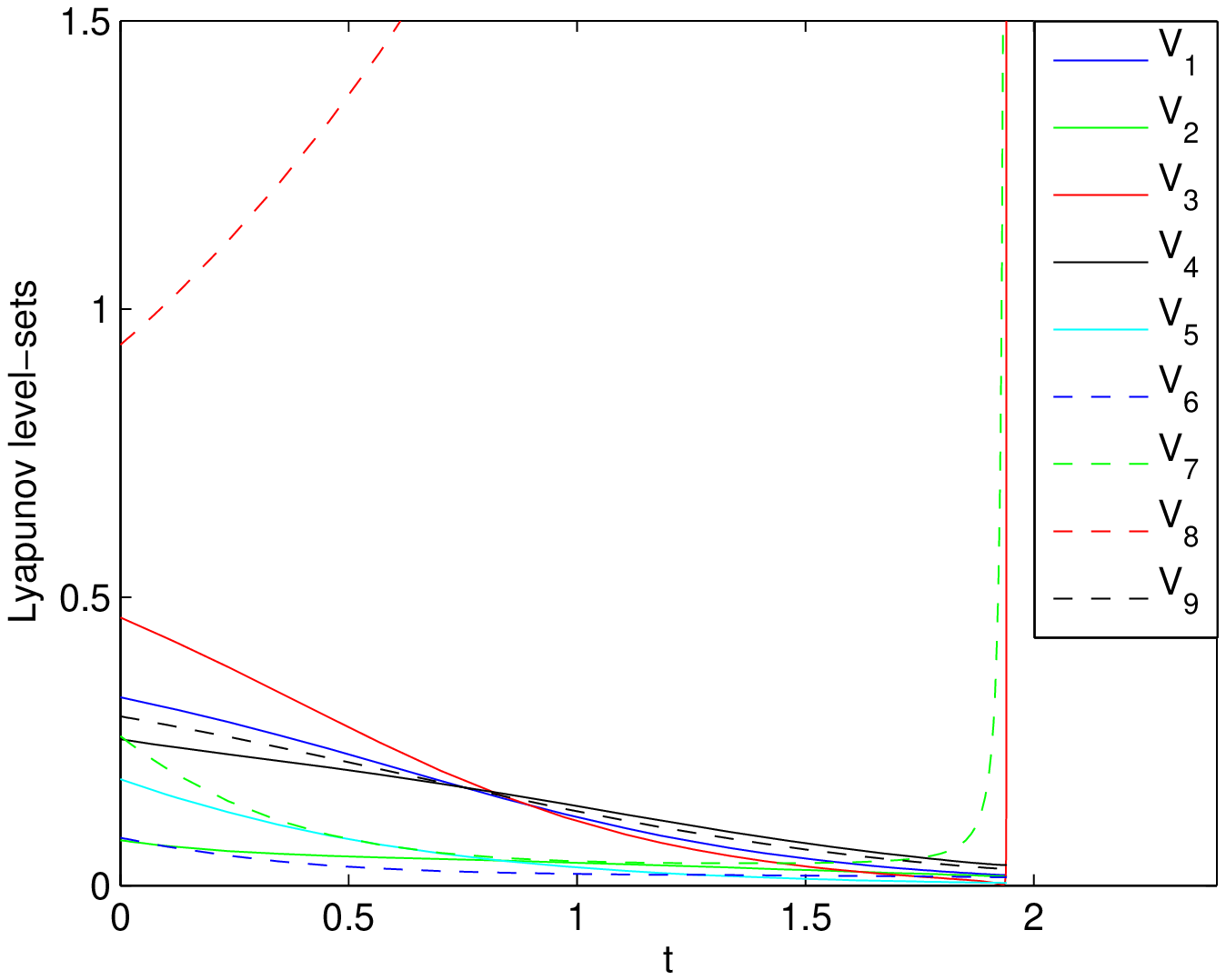}\label{F:level_42}
}
\caption[Optional caption for list of figures]{System states (selected) and Lyapunov functions starting from an unstable initial condition, without any control.}
\label{F:pre_control}
\end{figure*}
Figure\,\ref{F:pre_control} shows the evolution of the system state variables (belonging to subsystems $\mathcal{S}_1\,,\,\mathcal{S}_3\,,\,\mathcal{S}_4\,,\,\mathcal{S}_5\,,\,\mathcal{S}_7$ and $\mathcal{S}_8$) and the subsystem LFs, starting from an unstable initial condition. In particular, the state variables belonging to the subsystems $\mathcal{S}_3\,,\,\mathcal{S}_7$ and $\mathcal{S}_8$ `escape' to infinity while other subsystems remain reasonably bounded, over the shown time window.

Algorithm\,\ref{ALG:control} is used to compute distributed stabilizing linear controllers (with $\bar{U}_{i,1}=0~\forall i$), satisfying \eqref{E:control_SOS}. Table\,\ref{Tab:control} lists the results, while the trajectories after applying control are shown in Fig.\,\ref{F:post_control}. Interestingly, even though $\mathcal{S}_3$ was unbounded without control (Fig.\,\ref{F:pre_control}), the algorithm finds that there is actually no need for control in $\mathcal{S}_3$ provided its neighbors $\mathcal{S}_2$ and $\mathcal{S}_8$ remain bounded by their initial level-sets (Fig.\,\ref{F:post_control}). On the other hand, $\mathcal{S}_1$ and $\mathcal{S}_4$ apply control, although they were bounded for over $t\!\in\![0,2)$ without control (Fig.\,\ref{F:pre_control}).

The distributed control design is, however, conservative. For example, the maximum row-sum of the resulting comparison matrix (with control) is only marginally negative (Table\,\ref{Tab:control}), while its maximum eigenvalue  actually turns out to be $-0.06$. 

\begin{table}[hb]
\caption{Distributed Control Results}
\label{Tab:control}
\begin{center}
\begin{tabular}{|c|c|c|c|c|c|}
\hline
$i$ & $\gamma_i^0$ 	& $\sum_{j=1}^9 a_{ij}$ & $\sum_{j=1}^9 a_{ij}\gamma_j^0$ & $\bar{U}_{i,2}$  & $u_{t,i,2}\left(x_{i,1},\,x_{i,2}\right)$\\
\hline
1	 & 0.33	 & -0.000	 & -0.029	 & 0.17	 & $- 0.22\,x_{1,1} - 0.10\,x_{1,2}$	 \\ 
 \hline

2	 & 0.08	 & -0.140	 & -0.000	 & 0.00	 & ---	  \\ 
 \hline

3	 & 0.46	 & -0.016	 & -0.005	 & 0.00	 & ---	  \\ 
 \hline

4	 & 0.25	 & -0.000	 & -0.011	 & 0.12	 & $- 0.18\,x_{4,1} - 0.08\,x_{4,2}$	  \\ 
 \hline

5	 & 0.18	 & -0.048	 & -0.007	 & 0.00	 & ---	 \\ 
 \hline

6	 & 0.08	 & -0.101	 & -0.000	 & 0.00	 & ---  \\ 
 \hline

7	 & 0.26	 & -0.094	 & -0.000	 & 0.65	 & $- 0.47\,x_{7,1} - 0.89\,x_{7,2}$	  \\ 
 \hline

8	 & 0.94	 & -0.000	 & -0.109	 & 3.88	 & $- 0.56\,x_{8,1} - 2.89\,x_{8,2}$	  \\ 
 \hline

9	 & 0.29	 & -0.020	 & -0.005	 & 0.00	 & ---  \\ 
 \hline
\end{tabular}
\end{center}
\end{table}

 \begin{figure*}[thpb]
\centering
\subfigure[]{
\includegraphics[scale=0.5]{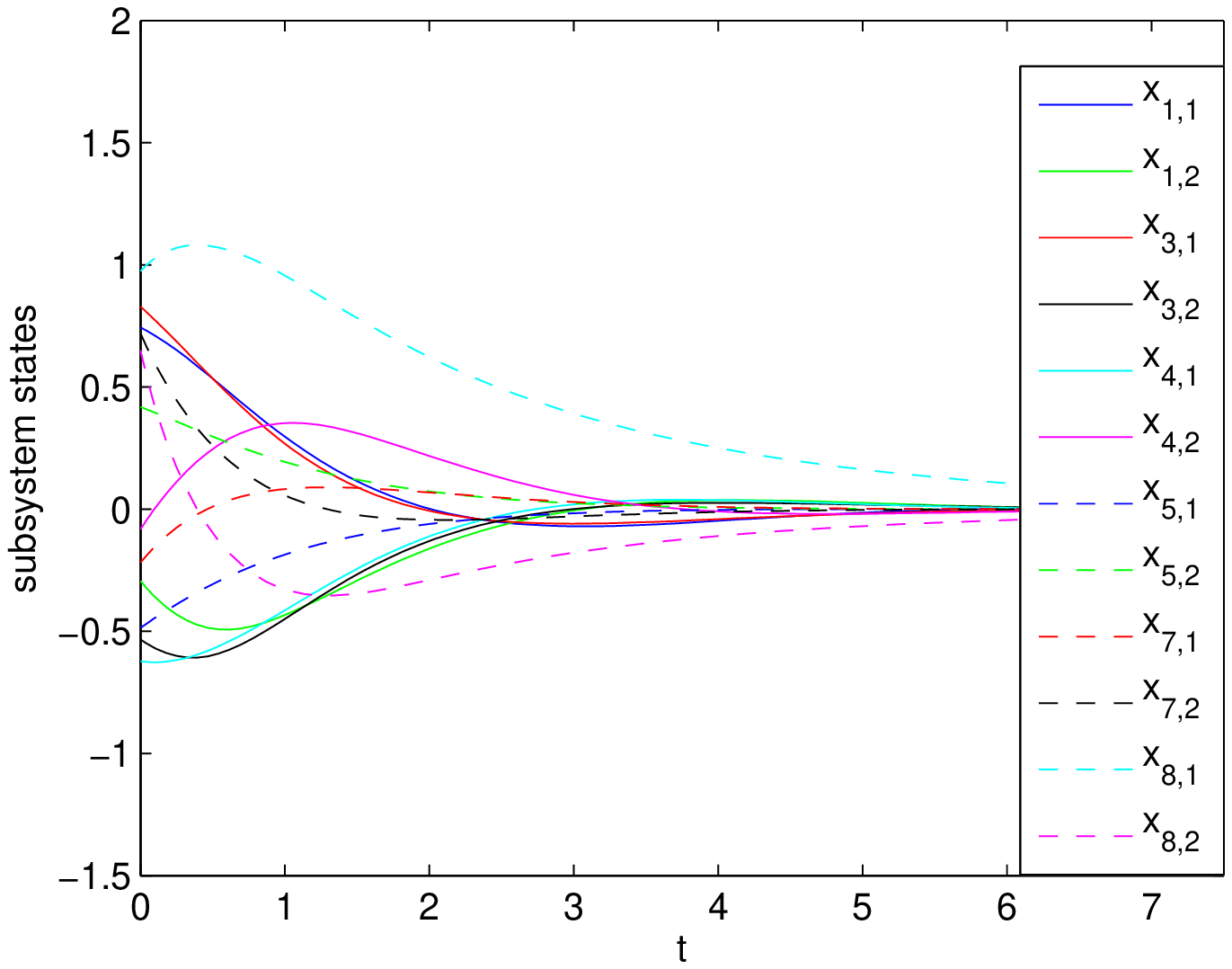}\label{F:states_42_control}
}
\subfigure[]{
\includegraphics[scale=0.5]{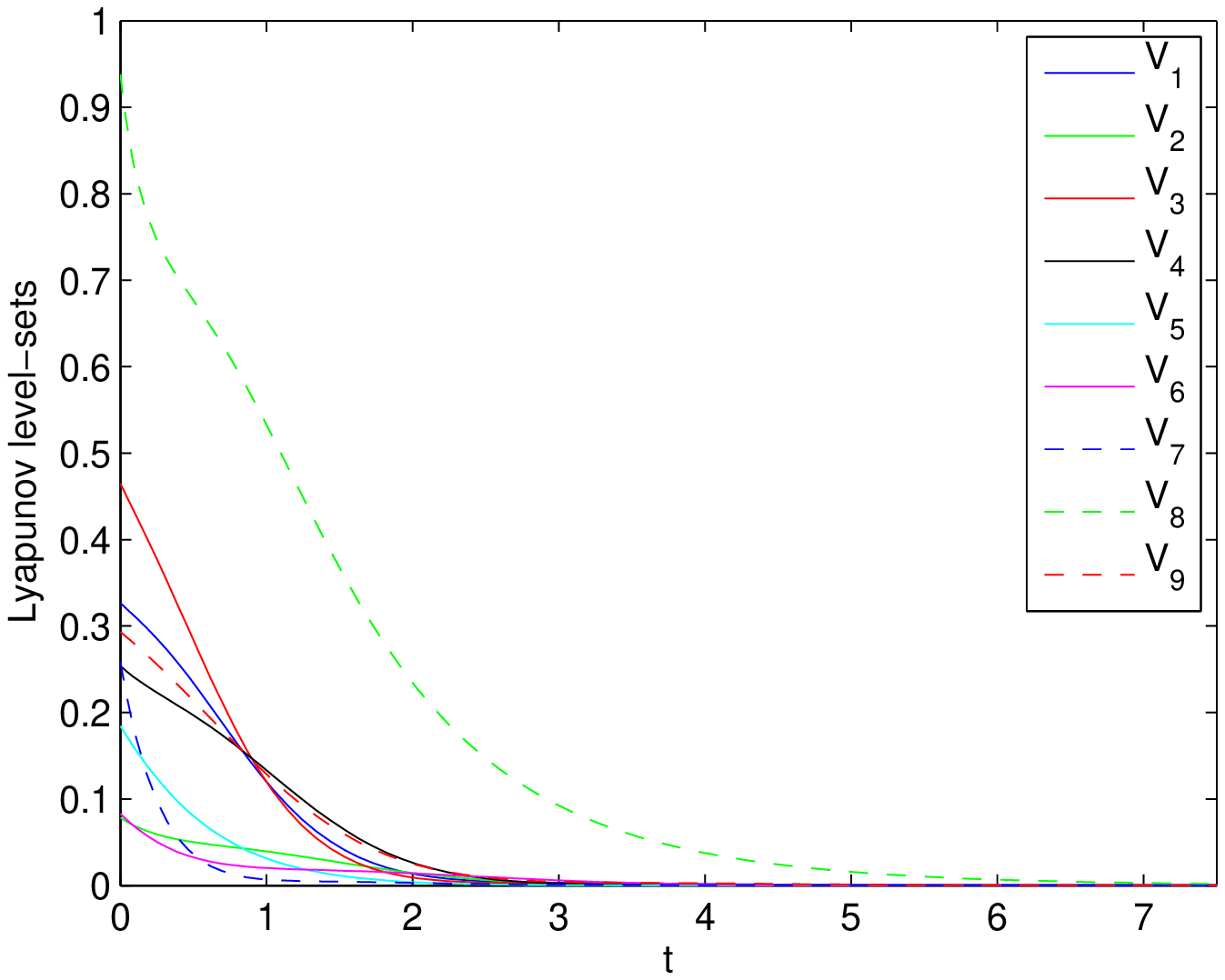}\label{F:level_42_control}
}
\caption[Optional caption for list of figures]{System states and Lyapunov functions with the same initial condition, after application of distributed stabilizing control.}
\label{F:post_control}
\end{figure*}

%==========================
\section{Conclusion}\label{S:concl}
%==========================
The paper presents a distributed control strategy in which agents (subsystems) coordinate with their immediate neighbors to compute optimal local control strategies that exponentially stabilize the full nonlinear network. The proposed algorithm can be easily scalable to very large-scale, sparse, interconnected systems. Future work will explore ways to make the algorithm less conservative. One such way is to use a hierarchical two-level multi-agent control scheme, where the agents exchange some minimal information with a higher-level central agent. The central agent can perform minimal computations such as checking if the comparison matrix is Hurwitz (instead of the diagonally-dominant condition). Higher order polynomials for the subsystem Lyapunov functions could be used for potentially improved control design. It would be interesting to apply the proposed algorithm on some real-world system models, such as a network preserving power system network.

%{\color{red}Mention: Multi-layered control architecture? Use higher order LF polynomials, with possible algorithmic improvements/changes?! Apply on power systems?!}

%\begin{ack}
%\end{ack}

\bibliography{references,ifacconf}             % bib file to produce the bibliography
                                                     % with bibtex (preferred)

%\appendix
%\section{}    % Each appendix must have a short title.

\end{document}